\documentclass[11pt,a4paper]{article}
\usepackage{amsfonts}
\usepackage{epsfig}
\usepackage{epstopdf}
\usepackage[T1]{fontenc}
\usepackage{geometry}
\usepackage{amsbsy,amsmath,latexsym,amsfonts, epsfig, color, authblk, amssymb, graphics, bm}
\usepackage{fancybox}
\usepackage{fancyhdr}
\usepackage{setspace}
\usepackage{algorithm}
\usepackage{algorithmic}
\usepackage{nccmath}
\usepackage{cases}
\usepackage[colorlinks, citecolor=blue]{hyperref}

\newenvironment{proof}{{\noindent \bf Proof:}}{\hfill$\Box$\medskip}

\newtheorem{theorem}{Theorem}[section]
\newtheorem{lemma}{Lemma}[section]
\newtheorem{definition}{Definition}[section]

\newtheorem{assumption}{Assumption}
\newtheorem{remark}{Remark}[section]

\definecolor{lred}{rgb}{1,0.8,0.8}
\definecolor{lblue}{rgb}{0.8,0.8,1}
\definecolor{dred}{rgb}{0.6,0,0}
\definecolor{dblue}{rgb}{0,0,0.5}
\definecolor{dgreen}{rgb}{0,0.5,0.5}

\begin{document}

\title{A Superlinear Convergence Framework for Kurdyka-{\L}ojasiewicz Optimization}


\author{Yitian Qian\footnote{(mayttqian@mail.scut.edu.cn) School of Mathematics, South China University of Technology.}\ \ {\rm and}\ \ Shaohua Pan\footnote{Corresponding author\,(shhpan@scut.edu.cn), School of Mathematics, South China University of Technology, Guangzhou.}}

\maketitle

\begin{abstract}
 This work extends the iterative framework proposed by Attouch et al. (in Math. Program. 137: 91-129, 2013) for minimizing a nonconvex and nonsmooth function $\Phi$ so that the generated sequence  possesses a Q-superlinear convergence rate. This framework consists of a monotone decrease condition, a relative error condition and a continuity condition, and the first two conditions both involve a parameter $p\!>0$. We justify that any sequence conforming to this framework is globally convergent when $\Phi$ is a Kurdyka-{\L}ojasiewicz (KL) function, and the convergence has a Q-superlinear rate of order $\frac{p}{\theta(1+p)}$ when $\Phi$ is a KL function of exponent $\theta\in(0,\frac{p}{p+1})$. Then, we illustrate that the iterate sequence generated by a $q\in[2,3]$-order regularized proximal Newton method for composite optimization problems with a nonconvex and nonsmooth term belongs to this framework, and consequently, first achieve the Q-superlinear convergence rate of order $4/3$ for a  cubic regularization method to solve this class of composite problems with KL property of exponent $1/2$. 
\end{abstract}


\medskip
\noindent
{\bf AMS:} 90C26; 47N10; 65K05

\section{Introduction}\label{sec1}

Let $\mathbb{X}$ denote a finite dimensional real vector space with the inner product $\langle\cdot,\cdot\rangle$ and its induced norm $\|\cdot\|$. Consider the nonconvex and nonsmooth problem 
\begin{equation}\label{prob0}
	\min_{x\in\mathbb{X}}\Phi(x),
\end{equation}
where $\Phi\!:\mathbb{X}\to\overline{\mathbb{R}}\!:=(-\infty,\infty]$ is a proper lower semicontinuous (lsc) and lower bounded function. 
We are interested in the following iterative framework: 
\begin{description}
 \item [{\bf H1.}] for each $k\in\mathbb{N}$, $\Phi(x^{k+1})+a\|x^{k+1}-x^k\|^{p+1}\le
	\Phi(x^k)$;
	
\item [{\bf H2.}] for each $k\in\mathbb{N}$, $\exists w^{k+1}\in\partial\Phi(x^{k+1})$ such that
	$\|w^{k+1}\|\le b\|x^{k+1}\!-\!x^{k}\|^{p}$;
	
\item[{\bf H3.}] there exists a convergent subsequence $\{x^{k_j}\}_{j\in\mathbb{N}}$ with limit $\widetilde{x}$ such that $\limsup_{j\to\infty}\Phi(x^{k_j})\le\Phi(\widetilde{x})$,
\end{description}
where $a>0, b\ge0$ and $p>0$ are the constants. Condition H1 restricts the monotone decrease magnitude of the objective values, condition H2 controls the relative error of $x^{k+1}$ to be a critical point of $\Phi$, i.e. a stationary point of \eqref{prob0}, and condition H3 is rather weak, which along with the lower semicontinuity of $\Phi$ implies that $\lim_{j\to\infty}\Phi(x^j)=\Phi(\widetilde{x})$. When $p=1$, the above iterative framework reduces to the popular one proposed by Attouch et al. \cite{Attouch13}. 
\subsection{Main Motivations}\label{sec1.1}

 For the sequences satisfying conditions H1-H3 for $p=1$, Attouch et al. \cite{Attouch13} proved its global convergence for the KL function $\Phi$ and its R-linear convergence rate for the KL function $\Phi$ with exponent $\theta\in(0,1/2]$. From \cite[Section 5]{Attouch13}, if the proximal gradient (PG) method with a monotone line-search is applied to the zero-norm regularized logistic regression, i.e., problem \eqref{prob0} with 
\begin{equation}\label{znorm-LG}
 \Phi(x)=\frac{1}{m}\sum_{i=1}^m\log\big[1+\exp(-b_i(a_i^{\mathbb{T}}x))\big]+\frac{\mu}{2}\|x\|^2+\lambda\|x\|_0\quad{\rm for}\ x\in\mathbb{R}^n
\end{equation}
where $\lambda>0$ and $\mu>0$ are the regularization parameters, and $a_i\in\mathbb{R}^n$ and $b_i\in\mathbb{R}$ for $i=1,2,\ldots,m$ are the given data, the generated iterate sequence comply with the above framework for $p=1$, and Figure \ref{fig1} below demonstrates that this sequence actually possesses a superlinear convergence rate. Moreover, such $\Phi$ is a KL function of exponent $1/2$ by \cite[Example 4.3]{WuPanBi21}. We find that a similar local convergence behavior also occurs when $\Phi$ takes the zero-norm regularized least squares function, another KL function of exponent $1/2$ by \cite[Example 4.3]{WuPanBi21}. This means that there is still room to improve the convergence rate result in \cite{Attouch13} when $\Phi$ is a KL function of exponent $\theta\in(0,1/2]$. Although Ochs also achieved a local convergence result in \cite{Ochs18} for the sequences satisying conditions H1-H3 with $p=1$, he did not give an improved local convergence rate. Inspired by this, we focus on the convergence analysis for the iterate sequence complying with conditions H1-H3 in this work. 
 \begin{figure}[H]
  \centering
  \includegraphics[width=12.0cm,height=5.0cm]{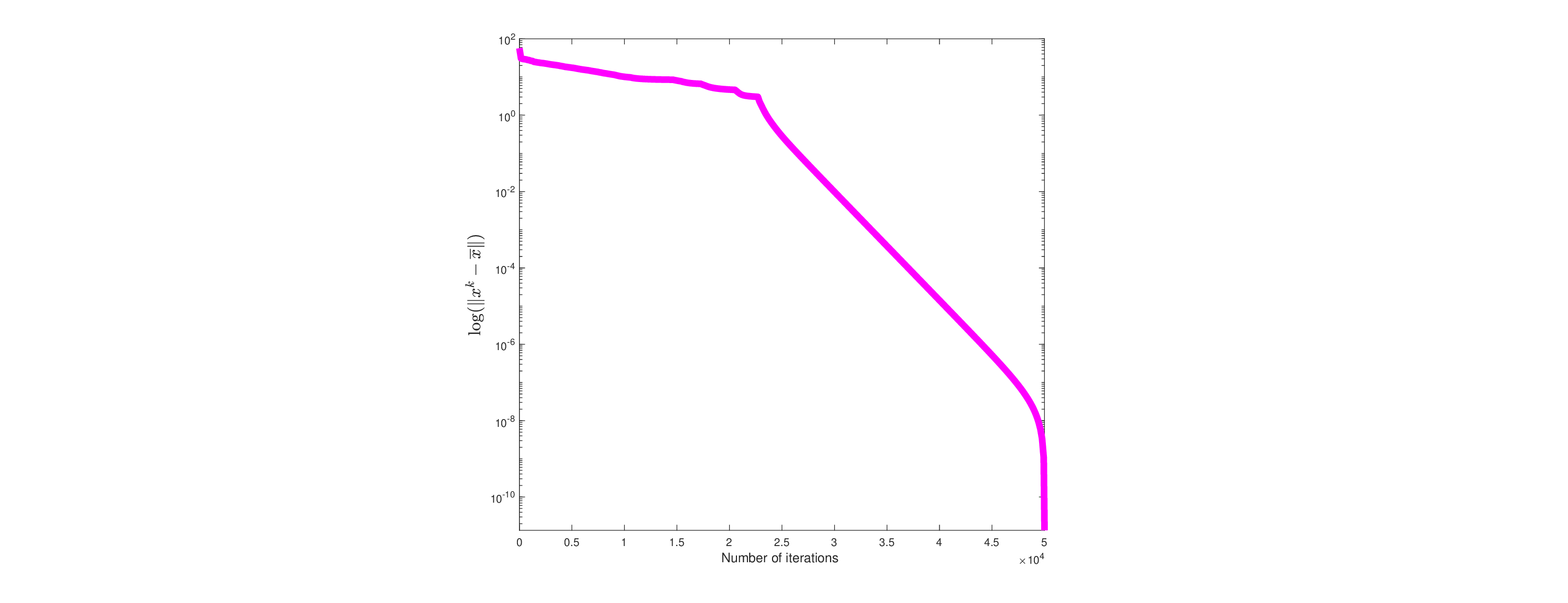}
  \caption{\small Convergence behavior of the iterates yielded by a PG method for minimizing $\Phi$ in \eqref{znorm-LG} with $\mu=10^{-5}$ and $\lambda=0.1$, where all $a_i\in\mathbb{R}^n$ and $b_i\in\mathbb{R}$ are generated randomly.}
 	\label{fig1}
 \end{figure}
 Another motivation is to establish the global convergence and Q-superlinear convergence rate of the iterate sequence generated by the cubic regularization (CR) method for the nonconvex and nonsmooth composite problem
\begin{equation}\label{composite}
	\min_{x\in\mathbb{X}}F(x):=f(x)+g(x),
\end{equation}
where $f,g\!:\mathbb{X}\to\overline{\mathbb{R}}$ are proper lsc functions and satisfy the basic assumption:
\begin{assumption}\label{ass-F} 
 \begin{description}	
 \item[(i)] $f$ is twice differentiable on an open set containing  ${\rm dom}g$; 
 
 \item[(ii)] the function $g$ is lower bounded, and is continuous relative to its domain. 	
 \end{description}
\end{assumption}
 The idea of cubic regularization first appeared in \cite{Griewank81} for solving problem \eqref{composite} with $g\equiv 0$, where Griewank proved that any accumulation point of the generated iterate sequence is a second-order critical point of $f$. Most of later works on the CR method were  carried out for this class of smooth optimization problems (see, e.g., \cite{Nesterov06,1Cartis11,2Cartis11,Yue19,Zhou18}). For example, Nesterov and Polyak \cite{Nesterov06} achieved the Q-superlinear convergence rate of the objective value sequence with order $4/3$ for gradient-dominated functions (a little stronger than KL functions of exponent $1/2$) and order $3/2$ for star-convex functions having a global nondegenerate optimal set (essentially star-convex functions with quadratic growth) respectively in Theorems 7 and 5 of that work, and obtained the global convergence and local Q-quadratic convergence rate of the iterate sequence under the local strong convexity of $f$ (see \cite[Theorem 3]{Nesterov06}); Cartis et al. \cite{1Cartis11} proposed an adaptive CR method by using a dynamic regularization parameter and solving the subproblems inexactly, and established the global convergence and local Q-superlinear convergence rate of the iterate sequence under the local strong convexity of $f$ (see \cite[Theorem 4.5 \& Corollary 4.8]{1Cartis11}); Yue et al. \cite{Yue19} proved that the iterate sequence generated by the CR method converges Q-quadratically to a second-order critical point under a local error bound condition, which does not require the isolatedness of stationary points and is much weaker than the nondegeneracy condition in \cite[Theorem 4.1]{Griewank81} and \cite[Theorem 3]{Nesterov06}; and Zhou et al. \cite{Zhou18} characterized the Q-superlinear convergence rate of the objective value sequence and the R-superlinear convergence rate of the iterate sequence for the KL function $f$ of exponent $(0,{2}/{3})$. To the best of our knowledge, there is no work to explore the global convergence and local convergence rate of the iterate sequence generated by the CR method for composite problems with nonconvex and nonsmooth terms, though some focus on the global iteration complexity for nonsmooth convex composite problems \cite{Nesterov08,Grapiglia19,Jiang20}. 

  Recently, several tensor methods (of course covering the CR method) were proposed for solving problem \eqref{composite} with a convex $g$ (see \cite{Nesterov22,Doikov22}) and its more general formulation (see \cite{Nabou23,Necoara21}). Among others, Doikov and Nesterov \cite{Doikov22} considered a regularized composite tensor method of degree $p\ge 2$, which becomes the CR method for $p=2$, and established the local Q-superlinear convergence rate of order $\frac{p}{q-1}$ for the objective value sequence and the norm of minimal subgradient sequence under the uniform convexity of degree $q\ge 2$ on $f$; Nesterov \cite{Nesterov22} investigated an inexact basic tensor method that uses approximate solutions of the auxiliary problems for nonsmooth convex composite problems, and analyzed its iteration complexity on the objective values; Necoara and Lupu \cite{Necoara21} proposed a general higher-order majorization-minimization algorithm framework for solving \eqref{prob0} by $p\,(\ge1)$ higher-order surrogate of $\Phi$ at each iterate, and obtained the superlinear convergence rate of the objective value sequence under the KL property of $\Phi$ with exponent $\theta\in(0,1/(p\!+\!1))$ (see \cite[Theorem 4.6]{Necoara21}); and Nabou and Necoara \cite{Nabou23} developed a general composite higher-order algorithmic framework for solving \eqref{composite} with $f=\varphi(F(\cdot))$ based on $p\,(\ge1)$ higher-order surrogates of $f$, and when higher-order composite surrogates are sufficiently smooth, they provided the global asymptotic stationary point guarantees and achieved the local R-linear convergence rate of objective value sequence for the KL objective function with exponent $\theta\in(0,p/(p\!+\!1)]$ (see \cite[Theorem 3]{Nabou23}). As far as we know, there are no local convergence rate results for the iterate sequences of the above tensor methods. 
\subsection{Our Contributions}\label{sec1.2}

 In this work, we investigate the global convergence and the local convergence rate for a sequence $\{x^k\}_{k\in\mathbb{N}}$ satisfying conditions H1-H3, and propose a $q\in[2,3]$-order 
 regularized proximal Newton method for problem \eqref{composite} to demonstrate that such a sequence $\{x^k\}_{k\in\mathbb{N}}$ is accessible. The main contributions are as follows.
 \begin{itemize}
  \item For the KL function $\Phi$, the sequence $\{x^k\}_{k\in\mathbb{N}}$ is proved to converge to a (limiting) critical point of $\Phi$, and the convergence is shown to have Q-superlinear rate of order $\frac{p}{\theta (p+1)}$ if $\Phi$ is a KL function of exponent  $\theta\in(0,\frac{p}{p+1})$ and the R-linear rate if $\Phi$ is a KL function of exponent $\theta=\frac{p}{p+1}$. This is the first local Q-superlinear convergence result for the iterate sequences of nonconvex and nonsmooth KL optimization problems, which implies that the sequences from the iterative framework of \cite{Attouch13} possess local Q-superlinear convergence rate of order $\frac{1}{2\theta}$ when $\Phi$ is a KL function of exponent $\theta\in(0,1/2)$, and  improves the convergence rate results there. As a byproduct, we obtain local Q-superlinear rate of order $\frac{p}{\theta (p+1)}$ for the sequence $\{\Phi(x^k)\}_{k\in\mathbb{N}}$ if $\Phi$ is a KL function of exponent  $\theta\in(0,\frac{p}{p+1})$. 
	
  \item  A $q\in[2,3]$-order regularized proximal Newton method is proposed for solving the nonconvex and nonsmooth composite problem \eqref{composite}, and the generated iterate sequence is proved to satisfy conditions H1-H3 with $p\!=q\!-\!1$. This method seeks in each step a stationary point of a $q$-order polynomial regularized by $g$, which is implementable in practice at least when the proximal mapping of $g$ has a closed form. When $F$ is a KL function of exponent $\theta\in(0,{(q\!-\!1)}/{q})$, the generated iterate sequence has Q-superlinear convergence rate with order ${(q\!-\!1)}/{(\theta q)}$. For $q=3$, we achieve local Q-superlinear convergence rate with order ${2}/{(3\theta)}$ for the iterate and objective value sequences yielded by the proposed CR method if $F$ is a KL function of exponent $\theta\in(0,{2}/{3})$, say, the zero-norm regularized least squares or logsitic regression, thereby extending the results in \cite{Zhou18} for smooth optimization problems to problem \eqref{composite}.
	
  \item As will be discussed in Remark \ref{remark-rate} (c), the iterate sequences generated by the tensor methods in \cite{Necoara21,Nabou23} comply with conditions H1-H3, so Theorem \ref{KL-rate} is applicable to them and improves the local convergence results there.  
\end{itemize}

\section{Preliminaries}\label{sec2}

 We recall the basic subdifferential of a function  $h\!:\mathbb{X}\to\overline{\mathbb{R}}$ at a point $x\in{\rm dom}h$, and refer the reader to standard texts such as \cite{RW98} and \cite{Mordu06} for details. 
\begin{definition}\label{Gsubdiff-def}
 (see \cite[Definition 8.3]{RW98}) Consider a function $h\!:\mathbb{X}\to\overline{\mathbb{R}}$ and a point $x\in{\rm dom}h$. The regular subdifferential of $h$ at $x$ is defined as
 \[
	\widehat{\partial}h(x):=\bigg\{v\in\mathbb{X}\,|\,
	\liminf_{x\ne x'\to x}\frac{h(x')-h(x)-\langle v,x'-x\rangle}{\|x'-x\|}\ge 0\bigg\},
 \]
 and its basic (known as the limiting or Morduhovich) subdifferential at $x$ is 
 \[
	\partial h(x):=\bigg\{v\in\mathbb{X}\,|\, \exists\,x^k\to x\ {\rm with}\ h(x^k)\to h(x)\ {\rm and}\
	v^k\in\widehat{\partial}h(x^k)\ {\rm with}\ v^k\to v\bigg\}.
 \]
\end{definition}

 We call $\overline{x}\in\mathbb{X}$ a critical point of a proper function $h\!:\mathbb{X}\to\overline{\mathbb{R}}$ if $0\in\partial h(x)$. 
To introduce the KL function of exponent $\theta\in[0,1)$, for any $\eta>0$, we denote by $\Upsilon_{\!\eta}$ the set of all continuous concave $\varphi\!:[0,\eta)\to\mathbb{R}_{+}$ that are continuously differentiable on $(0,\eta)$ with $\varphi(0)=0$ and $\varphi'(s)>0$ for all $s\in(0,\eta)$.
\begin{definition}\label{KL-Def}(see \cite{Attouch10})
 A proper function $h\!:\mathbb{X}\!\to\overline{\mathbb{R}}$ is said to have the KL property at $\overline{x}\in{\rm dom}\,\partial h$ if there exist $\eta\in(0,\infty]$, a neighborhood $\mathcal{U}$ of $\overline{x}$ and a function $\varphi\in\Upsilon_{\!\eta}$ such that for all $x\in\mathcal{U}\cap\big[h(\overline{x})<h<h(\overline{x})+\eta\big]$,
 \begin{equation}\label{KL-ineq}
 \varphi'(h(x)-h(\overline{x})){\rm dist}(0,\partial h(x))\ge 1.
 \end{equation}
  If $\varphi$ can be chosen as $\varphi(s)=cs^{1-\theta}$ with $\theta\in[0,1)$ for some $c>0$, then $h$ is said to have the KL property of exponent $\theta$ at $\overline{x}$. If $h$ has the KL property (of exponent $\theta$) at every point of ${\rm dom}\,\partial h$, then it is called a KL function (of exponent $\theta$).
\end{definition}
\begin{remark}\label{KL-remark}
 To verify that a proper lsc function is a KL function of exponent $\theta$, it suffices to prove that it has the KL property of exponent $\theta$ at all critical points because by \cite[Lemma 2.1]{Attouch10} it has this property at all noncritical points. On the calculation of the KL exponent, we refer the reader to \cite{LiPong18,YuLiPong21,WuPanBi21}.
\end{remark}
\section{Convergence Analysis}\label{sec3}
  
 By condition H1 and the lower boundedness of $\Phi$, we have the following result.
\begin{lemma}\label{lemma1-Phi}
 Let $\{x^k\}_{k\in\mathbb{N}}$ be a sequence satisfying condition H1. Then, the sequence $\{\Phi(x^{k})\}$ is nonincreasing and convergent, and ${\displaystyle\lim_{k\to\infty}}\|x^{k}-x^{k-1}\|=0$.
 \end{lemma}
\subsection{Global Convergence}\label{sec3.1}

 We are ready to justify the global convergence for any sequence $\{x^k\}_{k\in\mathbb{N}}$ 
satisfying conditions H1-H3 under the KL property of $\Phi$. Its proof is inspired by the arguments for \cite[Lemma 2.6 \& Theorem 2.9]{Attouch13}, we include it for completeness.
\begin{theorem}\label{KL-converge}
 Let $\{x^k\}_{k\in\mathbb{N}}$ be a sequence satisfying conditions H1-H3. If $\Phi$ is a KL function, then $\sum_{k=0}^\infty\|x^{k+1}-x^k\|<\infty$.
 \end{theorem}
 \begin{proof} Let $\overline{x}=\widetilde{x}$, the limit of the subsequence $\{x^{k_j}\}_{j\in\mathbb{N}}$ in H3 such that $\Phi(x^{k_j})\to\Phi(\widetilde{x})$ as $j\to\infty$, which is clearly a cluster point of $\{x^k\}_{k\in\mathbb{N}}$. Together with Lemma \ref{lemma1-Phi}, we deduce that $\Phi(x^k)\to \Phi(\overline{x})$ as $k\to\infty$ and $\Phi(x^k)\ge\Phi(\overline{x})$ for all $k\in\mathbb{N}$. By the given assumption, $\Phi$ has the KL property at $\overline{x}$, so there exist a constant $\eta>0$, a neighborhood $\mathcal{U}$ of $\overline{x}$, and a function $\varphi\in\Upsilon_{\!\eta}$ as in Definition \ref{KL-Def} such that for all $x\in\mathcal{U}\cap[\Phi(\overline{x})<\Phi<\Phi(\overline{x})+\eta]$, 
 \begin{equation}\label{ineq1-KLPhi}
  \varphi'(\Phi(x)-\Phi(\overline{x})){\rm dist}(0,\partial \Phi(x))\ge 1.
 \end{equation}
 Let $\delta>0$ be such that $\mathbb{B}(\overline{x},\delta)\subset\mathcal{U}$ and choose $\rho\in(0,\delta)$. If necessary, we shrink $\eta$ so that $\eta<a(\delta-\rho)^{p+1}$. 
 Let $\Gamma_{k,k+1}:=\varphi\big(\Phi(x^{k})-\Phi(\overline{x})\big)
 -\varphi\big(\Phi(x^{k+1})-\Phi(\overline{x})\big)$.
 Fix any $k\in\mathbb{N}$. We claim that if $\Phi(x^k)<\Phi(\overline{x})+\eta$ and $x^k\in\mathbb{B}(\overline{x},\rho)$, then 
 \begin{align}\label{temp-ineq31}
  \|x^{k+1}-x^k\|
  &\le\|x^k-x^{k-1}\|^{\frac{p}{p+1}}\big(b{a}^{-1}\Gamma_{k,k+1}\big)^{\frac{1}{p+1}}\\
  \label{temp-ineq32}
  &\le\frac{1}{2}\|x^k-x^{k-1}\|+2^pba^{-1}\Gamma_{k,k+1}
 \end{align}
 where the second inequality holds because $u^{\frac{p}{p+1}}v^{\frac{1}{p+1}}\le u+v$ for all $u,v\ge0$, implied by Young's inequality. Indeed, if $x^{k+1}=x^k$, inequality \eqref{temp-ineq31} holds trivially, so we only need to consider that $x^{k+1}\ne x^k$. Now from H1 we have $\Phi(x^k)>\Phi(x^{k+1})\ge \Phi(\overline{x})$ which, along with inequality \eqref{ineq1-KLPhi} and condition H2, shows that $w^{k}\ne 0$ and $x^{k}\ne x^{k-1}$. Using inequality \eqref{ineq1-KLPhi} and H2 again leads to  
 \begin{equation}\label{gap-ineq}
  \varphi'\big(\Phi(x^{k})-\Phi(\overline{x})\big)\ge\frac{1}{b\|x^{k}-x^{k-1}\|^p}.
 \end{equation}
 In addition, by the definition of $\Gamma_{k,k+1}$ and the concavity of $\varphi$ on $[0,\eta)$, 
 \begin{equation*}
 \Gamma_{k,k+1} \ge\varphi'\big(\Phi(x^{k})\!-\!\Phi(\overline{x})\big)(\Phi(x^{k})\!-\!\Phi(x^{k+1})).
 \end{equation*}
 Combining the last two inequalities with condition H1 shows that \eqref{temp-ineq31} holds. 
 Since $x^{k_j}\to\widetilde{x}=\overline{x}$ as $j\to\infty$ and $\Phi(x^k)\to \Phi(\overline{x})$, by Lemma \ref{lemma1-Phi} there exists $\overline{k}\in\mathbb{N}$ such that $\Phi(x^k)\in[\Phi(\overline{x}),\Phi(\overline{x})+\eta)$ for all $k\ge \overline{k}$, which along with the continuity of $\varphi$ on $[0,\eta)$ yields that
 \begin{equation}\label{temp-ineq33}
  \|x^{\overline{k}}-\overline{x}\|+\|x^{\overline{k}}-x^{\overline{k}-1}\|
  +2\Big(\frac{\Phi(x^{\overline{k}})-\Phi(\overline{x})}{a}\Big)^{\frac{1}{p+1}}
  +\frac{2^{p+1}b}{a}\varphi\big(\Phi(x^{\overline{k}})-\Phi(\overline{x})\big)<\rho.
 \end{equation}
 Then, by using the above inequalities \eqref{temp-ineq32}-\eqref{temp-ineq33}, we can prove that for each $\nu>\overline{k}$, 
 \begin{align}\label{aim-ineq31}
  x^{\nu}\in\mathbb{B}(\overline{x},\rho),\qquad\qquad\qquad\qquad\qquad\\
  \label{aim-ineq32}
  \sum_{k=\overline{k}}^{\nu}\big\|x^{k+1}-x^{k}\big\|
  \le\sum_{k=\overline{k}}^{\nu}\frac{1}{2}\|x^k-x^{k-1}\|
  +\frac{2^pb}{a}\varphi\big(\Phi(x^{\overline{k}})\!-\!\Phi(\overline{x})\big).	
 \end{align}
 Indeed, when $\nu=\overline{k}+1$, from inequality \eqref{temp-ineq33} we have $x^{\overline{k}}\in\mathbb{B}(\overline{x},\rho)$, which along with condition H1 and $\Phi(x^{\overline{k}+1})>\Phi(\overline{x})$ implis that 
 \[
  \|x^{\overline{k}+1}-\overline{x}\|
  \le\|x^{\overline{k}+1}-x^{\overline{k}}\|+\|x^{\overline{k}}-\overline{x}\|
 \le  \|x^{\overline{k}}-\overline{x}\|+\Big(\frac{\Phi(x^{\overline{k}})-\Phi(\overline{x})}{a}\Big)^{\frac{1}{p+1}}<\rho.
 \]
 where the last inequality is due to \eqref{temp-ineq33}. This shows that \eqref{aim-ineq31} holds with $\nu=\overline{k}+1$. Summing \eqref{temp-ineq32} from $\overline{k}$ to $\overline{k}+1$ and using the nonnegativity of $\varphi$ yields that \eqref{aim-ineq32} holds with $\nu=\overline{k}+1$. Suppose that \eqref{aim-ineq31} and \eqref{aim-ineq32} hold some $\nu>\overline{k}$. Notice that inequality \eqref{aim-ineq32} implies  the following one
 \begin{equation}\label{temp-ineq3}
  \frac{1}{2}\sum_{k=\overline{k}}^{\nu}\big\|x^{k+1}-x^{k}\big\|
  \le\frac{1}{2}\|x^{\overline{k}}-x^{\overline{k}-1}\|
 +\frac{2^pb}{a}\varphi\big(\Phi(x^{\overline{k}})\!-\!\Phi(\overline{x})\big).
 \end{equation}
 Together with $\|x^{\nu+1}-\overline{x}\|\le \|x^{\overline{k}}-\overline{x}\|+\sum_{j=\overline{k}}^{\nu}\|x^{j+1}-x^j\|$ and \eqref{temp-ineq33}, we obtain
 \begin{align*}
  \|x^{\nu+1}-\overline{x}\|
 \le \|x^{\overline{k}}-\overline{x}\|+\|x^{\overline{k}}-x^{\overline{k}-1}\|+\frac{2^{p+1}b}{a}\varphi\big(\Phi(x^{\overline{k}})\!-\!\Phi(\overline{x})\big)<\rho.
 \end{align*}
 This shows that \eqref{aim-ineq31} holds for $\nu+1$. Summing inequality \eqref{temp-ineq32} from $\overline{k}$ to $\nu+1$ and using the nonnegativity of $\varphi$ yields that \eqref{aim-ineq32} holds for $\nu+1$. The above arguements show that inequalities \eqref{aim-ineq31}-\eqref{aim-ineq32} hold for $\nu>\overline{k}$. Note that inequality \eqref{aim-ineq32} implies \eqref{temp-ineq3}.
 Passing the limit $\nu\to\infty$ to the both sides of \eqref{temp-ineq3} yields the desired result. The proof is then completed. 
\end{proof}
\subsection{Local Convergence Rates}\label{sec3.2}

 We next establish the convergence rate of a sequence satisfying conditions H1-H3 under the assumption that $\Phi$ is a KL function of exponent $\theta\in(0,1)$. 
 \begin{theorem}\label{KL-rate}
  Let $\{x^k\}_{k\in\mathbb{N}}$ be a sequence complying with conditions H1-H3. 	Suppose that $\Phi$ is a KL function of exponent $\theta\in(0,1)$. Then, the sequence $\{x^k\}_{k\in\mathbb{N}}$ is convergent with limit $\overline{x}$, and furthermore,
  \begin{description}
   \item[(i)] when $\theta\in\!(0,\frac{p}{p+1})$, with any given $\varepsilon\in(0,1)$, for all sufficiently large $k$,
		\[
		\|x^{k+1}-\overline{x}\|\le\varepsilon\|x^{k}-\overline{x}\|^{\frac{p}{\theta(1+p)}};
		\]
		
   \item[(ii)] when $\theta\in[\frac{p}{p+1},1)$, there exist $\gamma>0$ and $\varrho\in(0,1)$ such that for all $k$ large enough,
		\begin{equation}\label{aim-ineq-rate}
			\|x^k-\overline{x}\|\le\sum_{j=k}^{\infty}\|x^{j+1}\!-\!x^{j}\|
			\le\left\{\begin{array}{cl}
				\gamma\varrho^{k} &{\rm if}\ \theta=\frac{p}{p+1},\\
				\gamma k^{\frac{1-\theta}{1-(1+1/p)\theta}}&{\rm if}\ \theta\in(\frac{p}{p+1},1).
			\end{array}\right.
		\end{equation}
 \end{description}
 \end{theorem}
\begin{proof} 
 For each $k$, write $\Delta_k\!:=\!\sum_{j=k}^{\infty}\|x^{j+1}\!-\!x^{j}\|$. By Theorem \ref{KL-converge}, $\Delta_k<\infty$ and the sequence $\{x^k\}_{k\in\mathbb{N}}$ is convergent, so it suffices to prove that parts (i) and (ii) hold. If there exists $\mathbb{N}\ni\widetilde{k}\ge\overline{k}$ such that $\Phi(x^{\widetilde{k}})=\Phi(\overline{x})$, where $\overline{k}\in\mathbb{N}$ is the same as in the proof of Theorem \ref{KL-converge}, then condition H1 implies that $x^{\widetilde{k}+1}=x^{\widetilde{k}}$. By induction, we have $x^k=x^{\widetilde{k}}$ for all $k\ge\widetilde{k}$, and the result follows. Thus, by Lemma \ref{lemma1-Phi}, it suffices to consider that $\Phi(x^{k})>\Phi(\overline{x})$ for all $k\in\mathbb{N}$. Note that \eqref{gap-ineq} holds with $\varphi(t)=ct^{1-\theta}\ (c>0)$ for $t\ge 0$, i.e., 
 \begin{equation}\label{ineq-Phi}
  (\Phi(x^{k})\!-\!\Phi(\overline{x}))^{\theta}\le bc(1-\theta)\|x^k-x^{k-1}\|^p
  \quad{\rm for\ all}\ k\ge\widetilde{k}.
 \end{equation}
 Note that $\Phi(x^{k})-\Phi(x^{k+1})\le\Phi(x^{k})-\Phi(\overline{x})$ for each $k\in\mathbb{N}$. Together with condition H1 and inequality \eqref{ineq-Phi}, for all $k\ge\widetilde{k}$, it holds that
  \begin{align}\label{temp-ineq4}
   \|x^{k+1}-x^k\|\le \Big(\frac{bc(1-\theta)}{a^{\theta}}\Big)^{\frac{1}{\theta (p+1)}}\|x^k-x^{k-1}\|^{\frac{p}{\theta(1+p)}}.
  \end{align}
  We proceed the arguments by the two cases: $\theta\in(0,\frac{p}{p+1})$ and $\theta\in[\frac{p}{p+1},1)$.
	
  \noindent
 {\bf Case 1: $\theta\in(0,\frac{p}{p+1})$.} 
 Let $M\!:=\!(\frac{bc(1-\theta)}{a^{\theta}})^{\frac{1}{\theta(p+1)}}$ and $\beta\!:=\!\frac{p}{\theta(1+p)}$. From the recursion relation in \eqref{temp-ineq4}, for any $j\ge k\ge\widetilde{k}$, 
  \begin{align*}
  \|x^{j+1}-x^{j}\|
  &\le M^{\sum_{t=0}^{j-k}\beta^t}\|x^k-x^{k-1}\|^{\beta^{j-k+1}}
		=M^{\frac{\beta^{j-k+1}-1}{\beta-1}}\|x^k-x^{k-1}\|^{\beta^{j-k+1}}\\
  &=M\|x^{k}-x^{k-1}\|^{\beta}\big(M^\frac{1}{\beta-1}\|x^{k}-x^{k-1}\|\big)^{\beta^{j-k+1}-\beta}.
 \end{align*}
 Since $\lim_{k\to\infty}\|x^{k}\!-\!x^{k-1}\|=0$ by Lemma \ref{lemma1-Phi}, there exists a sufficiently small $\epsilon_1\in(0,1/2)$ such that   $M^\frac{1}{\beta-1}\|x^{k}\!-\!x^{k-1}\|\le\!\epsilon_1$ for all $k\ge\widetilde{k}$ (if necessary by increasing $\widetilde{k}$). Along with the last inequality, for any $j\ge k\ge\widetilde{k}$, it holds that 
 $\|x^{j+1}-x^{j}\|\le M\epsilon_1^{\beta^{j-k+1}-\beta}\|x^{k}-x^{k-1}\|^{\beta}$. Summing this inequality from $j=k\ge\widetilde{k}$ to any $\nu\ge k$ and passing the limit $\nu\to\infty$ yields that 
 \[
   \|x^{k+1}-\overline{x}\|\le\sum_{j=k+1}^{\infty}\|x^{j+1}-x^j\|
   \le M\|x^{k+1}\!-\!x^{k}\|^{\beta}\sum_{j=k+1}^{\infty}\epsilon_1^{\beta^{j-k+1}-\beta}.
 \] 
 Since $\beta>1$, if necessary by shrinking $\epsilon_1$,  	$\sum_{j=2}^{\infty}\epsilon_1^{\beta^{j}-\beta^2}\le\sum_{j=2}^{\infty}0.5^{\beta^{j}-\beta^2}$ and $\epsilon_1^{\beta^2-\beta}\sum_{j=2}^{\infty}0.5^{\beta^{j}-\beta^2}\le\frac{\varepsilon}{2^{\beta+1}M}$, and then $\sum_{j=2}^{\infty}\epsilon_1^{\beta^{j}-\beta}=\epsilon_1^{\beta^2-\beta}\sum_{j=2}^{\infty}\epsilon_1^{\beta^{j}-\beta^2}\le \frac{\varepsilon}{2^{\beta+1}M}$. Together with the last inequality, it follows that
 \begin{align*}
  &\|x^{k+1}-\overline{x}\|\le M\|x^{k+1}-x^{k}\|^{\beta}
		\sum_{j=2}^{\infty}\epsilon_1^{\beta^{j}-\beta}\le 2^{-(\beta+1)}\varepsilon\|x^{k+1}-x^k\|^{\beta}\\
  &\le 2^{-(\beta+1)}\varepsilon\big(\|x^{k+1}-\overline{x}\|+\|x^k-\overline{x}\|\big)^{\beta}
  \le 0.5\varepsilon\big(\|x^{k+1}-\overline{x}\|^{\beta}+\|x^k-\overline{x}\|^{\beta}\big).
 \end{align*}  
 Note that $\lim_{k\to\infty}\|x^{k+1}-\overline{x}\|^{\beta-1}=0$. If necessary by increasing $k$, we assume that $0.5\varepsilon\|x^{k+1}-\overline{x}\|^{\beta-1}<\frac{1}{2}$. From the last inequality, it immediately follows that $0.5\|x^{k+1}-\overline{x}\|\le 0.5\varepsilon\|x^k-\overline{x}\|^{\beta}$, and the desired inequality follows.
	
 \noindent
 {\bf Case 2: $\theta\in[\frac{p}{p+1},1)$.} By the definition of $\Delta_k$ and the triangle inequality, $\|x^k-\overline{x}\|\le\Delta_k$, so we only need to establish the second inequality. 
	
 \noindent
 {\bf Subcase 2.1: $\theta=\frac{p}{p+1}$.} From \eqref{temp-ineq31}, 	$\|x^{k+1}-x^k\|\le\|x^k-x^{k-1}\|^{\frac{p}{p+1}}\big(\frac{b}{a}\Gamma_{k,k+1}\big)^{\frac{1}{p+1}}$ for all $k\ge\widetilde{k}$. By this recursion relation, if there exists $\widetilde{k}_1\ge\widetilde{k}$ such that $\|x^{\widetilde{k}_1}-x^{\widetilde{k}_1-1}\|=0$, then $\|x^k-x^{k-1}\|=0$ for all $k\ge\widetilde{k}_1$, and the conclusion then follows. Hence, it suffices to consider that $\|x^k-x^{k-1}\|>0$ for all $k\ge\widetilde{k}$. 
 Now summing \eqref{temp-ineq32} from any $k\ge\widetilde{k}$ to any $\nu>k$ yields that 
 \begin{align*}
 \frac{1}{2}\sum_{j=k}^{\nu}\big\|x^{j+1}-x^{j}\big\|
  &\le \frac{1}{2}\|x^k-x^{k-1}\|
		+\frac{2^pb}{a}\varphi\big(\Phi(x^{k})\!-\!\Phi(\overline{x})\big)\nonumber\\
  &=\frac{1}{2}\|x^k-x^{k-1}\|
		+\frac{2^pb}{a}c\big(\Phi(x^{k})\!-\!\Phi(\overline{x})\big)^{1-\theta}\nonumber\\
  &\le \frac{1}{2}\|x^k-x^{k-1}\|+\frac{2^pb}{a}c\big[bc(1-\theta)\big]^{\frac{1-\theta}{\theta}}
		\big\|x^k-x^{k-1}\big\|^{\frac{p(1-\theta)}{\theta}}\\
  &=\Big[\frac{1}{2}+2^pba^{-1}c[bc(1-\theta)]^{\frac{1-\theta}{\theta}}\Big]\|x^k-x^{k-1}\|\nonumber
 \end{align*}
 where the second inequality is due to \eqref{ineq-Phi}, and the second equality is using $\theta=\frac{p}{p+1}$. Write $M_1:=2^pba^{-1}c[bc(1-\theta)]^{\frac{1-\theta}{\theta}}$. Passing the limit $\nu\to\infty$ to the last inequality gives $\frac{1}{2}\Delta_k\le \big(\frac{1}{2}\!+M_2\big)(\Delta_{k-1}\!-\Delta_{k})$ and then $\Delta_{k}\le \frac{1+2M_2}{2(1+M_2)}\Delta_{k-1}$. By invoking this recursion, we obtain that  $\Delta_{k}\le\big(\frac{1+2M_2}{2(1+M_2)}\big)^{k-\widetilde{k}}\Delta_{\widetilde{k}}$. Thus, there exists $\gamma>0$ such that $\Delta_{k}\le\gamma\varrho^k$ with 
	$\varrho=\frac{1+2M_2}{2(1+M_2)}\in(0,1)$. 
	
   \noindent
   {\bf Subcase 2.2: $\theta\in(\frac{p}{p+1},1)$.} From the proof of subcase 2.1, for any $\nu>k\ge\widetilde{k}$, 
   \begin{align*}
	\frac{1}{2}\sum_{j=k}^{\nu}\big\|x^{j+1}-x^{j}\big\|
	&\le \frac{1}{2}\|x^k-x^{k-1}\|+\frac{2^pb}{a}c\big[bc(1-\theta)\big]^{\frac{1-\theta}{\theta}}
		\big\|x^k-x^{k-1}\big\|^{\frac{p(1-\theta)}{\theta}}\\
	&\le\Big[\frac{1}{2}+2^pba^{-1}c[bc(1-\theta)]^{\frac{1-\theta}{\theta}}\Big]\|x^k-x^{k-1}\|^{\frac{p(1-\theta)}{\theta}}\\
	&\le\Big[\frac{1}{2}+2^pba^{-1}c[bc(1-\theta)]^{\frac{1-\theta}{\theta}}\Big]\big(\Delta_{k-1}-\Delta_k\big)^{\frac{p(1-\theta)}{\theta}}
	\end{align*}
	where the second inequality is due to $\frac{p(1-\theta)}{\theta}<1$ and $\|x^k-x^{k-1}\|<1$ for $k\ge\widetilde{k}$. Passing the limit $\nu\to\infty$ to the last inequality, for any $k\ge\widetilde{k}$, we have
	\[
	\Delta_{k}^{\frac{\theta}{p(1-\theta)}}
	\le M_2^{\frac{\theta}{p(1-\theta)}}\big(\Delta_{k-1}-\Delta_{k}\big)
	\ \ {\rm with}\ M_2:=1+2^{p+1}ba^{-1}c(bc(1-\theta))^{\frac{1-\theta}{\theta}}.
	\]
	By using this resursion formula and following the same analysis technique as
	in \cite[Page 14]{Attouch09}, there exist $\widehat{k}\ge\widetilde{k}$
	and $\widetilde{\gamma}>0$ such that for all $k\ge\widehat{k}$,
	\[
	\Delta_{k}^{\mu}-\Delta_{k-1}^{\mu}\ge\widetilde{\gamma}>0\ \ {\rm with}\ \
	\mu=\frac{1-(1+p^{-1})\theta}{1-\theta}.
	\]
	Summing this inequality from $\widehat{k}$ to some $\nu\ge\widehat{k}$ yields
	that $\Delta_{\nu}^{\mu}-\Delta_{\widetilde{k}}^{\mu}\ge
	\widetilde{\gamma}(\nu-\widetilde{k})$. Since $\mu<0$, 
	there exists $\widetilde{\gamma}_1>0$ such that
	\(
	\Delta_{\nu}\le\big(\Delta_{\widehat{k}}^{\mu}\!+\widetilde{\gamma}(\nu-\widetilde{k})\big)^{1/\mu}
	\le \widetilde{\gamma}_1(\nu-\widehat{k})^{\frac{1-\theta}{1-(1+1/p)\theta}}.
	\)
  Consequently, the desired conclusion holds. 
 \end{proof}
 \begin{remark}\label{remark-rate}
 {\bf(a)} From Theorem \ref{KL-rate}, we conclude that the sequences conforming to conditions H1-H3 with $p=1$ have a Q-superlinear convergence rate of order ${1}/{(2\theta)}$ if $\Phi$ is a KL function of exponent $\theta\in(0,{1}/{2})$. This improves greatly the convergence result obtained in \cite[Theorem 2]{Attouch09}.		
	
 \noindent
 {\bf(b)} By combining \eqref{ineq-Phi} with condition H1, it is easy to get that for all $k\ge\widetilde{k}$, 
 \[
	\Phi(x^k)-\Phi(\overline{x})\le\Big[\frac{bc(1-\theta)}{a^{\frac{p}{p+1}}}\Big]^{\frac{1}{\theta}}\big[\Phi(x^{k-1})-\Phi(x^k)\big]^{\frac{p}{\theta(p+1)}}.
 \]
 Note that $\Phi(x^{k-1})-\Phi(x^k)\le\Phi(x^{k-1})-\Phi(\overline{x})$ for all $k\ge\widetilde{k}$. The last inequality implies that the sequence $\{\Phi(x^k)\}_{k\in\mathbb{N}}$ converges to $\Phi(\overline{x})$ with a Q-superlinear rate of order $\frac{p}{\theta(p+1)}$ when $\Phi$ is a KL function of exponent $\theta\in(0,\frac{p}{p+1})$. For $\theta\in[\frac{p}{p+1},1)$, by letting $\Delta_k:=\Phi(x^k)-\Phi(\overline{x})$, the last inequality is rewritten as 
 \[
 \Delta_k\le M_3(\Delta_{k-1}-\Delta_k)^{\frac{p}{\theta(p+1)}}\ \ {\rm with}\ \ M_3=\Big[\frac{bc(1-\theta)}{a^{\frac{p}{p+1}}}\Big]^{\frac{1}{\theta}}
 \ \ {\rm for\ all}\ k\ge\widetilde{k}.
 \]
 Using this recursion formula and following the same analysis technique as in \cite[Page 14]{Attouch09} shows that the sequence $\{\Phi(x^k)\}_{k\in\mathbb{N}}$ converges to $\Phi(\overline{x})$ with a Q-linear rate if $\Phi$ is a KL function of exponent $\theta=\frac{p}{p+1}$ and with a sublinear rate if $\Phi$ is a KL function of exponent $\theta\in(\frac{p}{p+1},1)$. 
 In addition, from condition H2 and inequality \eqref{temp-ineq4}, for all $k\ge\widetilde{k}$, with $C=\frac{bc(1-\theta)}{a^{\theta}}$ and $\beta=\frac{p}{\theta (p+1)}$, 
 \begin{align*}
 {\rm dist}(0,\partial\Phi(x^{k+1}))&\le b\|x^{k+1}-x^k\|^p
 \le bC^{\beta}\|x^k-x^{k-1}\|^{p\beta}\\
 &\le bC^{\sum_{j=1}^{k-\widetilde{k}+1}\beta^j}\|x^{\widetilde{k}}-x^{\widetilde{k}-1}\|^{p\beta^{k-\widetilde{k}+1}}\\
 &=bC^{\frac{1}{1-\beta}}\big[C^{\frac{1}{\beta-1}}\|x^{\widetilde{k}}-x^{\widetilde{k}-1}\|^{p}\big]^{\beta^{k-\widetilde{k}+1}}
 \le bC^{\frac{1}{1-\beta}}({1}/{2})^{\beta^{k-\widetilde{k}+1}}
 \end{align*}
 where the last inequality is due to  $C^{\frac{1}{\beta-1}}\|x^{\widetilde{k}}\!-\!x^{\widetilde{k}-1}\|^{p}\le 0.5$ (if necessary by increasing $\widetilde{k}$) implied by $\lim_{k\to\infty}\|x^{k}-x^{k-1}\|=0$. This implies that the sequence $\{{\rm dist}(0,\partial\Phi(x^{k}))\}_{k\in\mathbb{N}}$ converges to $0$ with a R-superlinear rate of order $r\in(1,\frac{p}{\theta(p+1)}]$ when $\Phi$ is a KL function of exponent $\theta\in(0,\frac{p}{p+1})$. 
	
 \noindent
 {\bf(c)} From the proofs of \cite[Theorems 4.3 $\&$ 4.4]{Necoara21}, the sequence  $\{x^k\}_{k\in\mathbb{N}}$ generated by GHOM, a general higher-order majorization-minimization algorithm, conforms to conditions H1-H3. Then, by invoking Theorem \ref{KL-rate}, the sequence generated by GHOM possesses local Q-superlinear rate of order $\frac{p}{\theta(p+1)}$ if the cost function $f$ in \cite[Equa (1.1)]{Necoara21} is a KL function of exponent $\theta\in(0,\frac{p}{p+1})$. In addition, from \cite[Theorem 1 $\&$ Equa (22)]{Nabou23}, the sequence $\{x^k\}_{k\in\mathbb{N}}$ generated by GCHO, a general composite higher-order algorithm, also satisfies conditions H1-H3, so has local Q-superlinear rate of order $\frac{p}{\theta(p+1)}$ by Theorem \ref{KL-rate} when the objective function $f$ there is a KL function of exponent $\theta\in(0,\frac{p}{p+1})$.
\end{remark}
\section{A $q$-order Regularized Proximal Newton Method}\label{sec4}
 
 As mentioned in the introduction, the past decade has witnessed active research on proximal Newton methods for composite convex optimization problems (see, e.g., \cite{Lee14,Yue191,Mordu22,Kanzow21}), to the best of our knowledge, there is no work to achieve the local superlinear convergence rate on proximal Newton methods for composite optimization with a nonconvex and nonsmooth term. In this section, we propose a $q\in[2,3]$-order regularized proximal Newton (RPNT) method for solving the nonconvex and nonsmooth composite problem \eqref{composite}, and prove that the generated iterate sequence comply to conditions H1-H3. 
 
 For each $k$, let $f_k\!:\mathbb{X}\to\mathbb{R}$ be the quadratic expansion of $f$ at $x^k\!\in{\rm dom}\,g$: 
 \begin{equation}\label{ellk}  f_k(x):=f(x^k)+\langle\nabla\!f(x^{k}),x\!-\!x^k\rangle+\frac{1}{2}\langle x\!-\!x^k,\nabla^2\!f(x^{k})(x\!-\!x^k)\rangle\ \ \forall x\in\mathbb{X}.
 \end{equation}  
 In each iterate, our $q\in[2,3]$-order RPNT method seeks a stationary point $y^k$ of the following subproblem with $F_k(y^k)\le F(x^k)$:
 \begin{equation}\label{subprobk}
 \min_{x\in\mathbb{X}}F_k(x):=f_k(x)+g(x)+(L_k/q)\|x-x^k\|^q,
 \end{equation}    
 where $L_k>0$ is a parameter, and its iterate steps are described as follows. For $q=3$, it reduces to a CR method, and for $q=2$ it becomes a RPNT method.
 \begin{algorithm}[h]
 \caption{\label{cubic-reg}{\bf\,($q$-order RPNT method)}}
  \textbf{Initialization:} Choose $\epsilon>0,\,q\in\![2,3],\,0<\!L_{\min}<\!L_{\max},\,\tau>1,\,\delta\in(0,1)$, and an initial point $x^0\in{\rm dom}\,g$. Set $k:=0$.\\
  \textbf{while} ${\rm dist}(0,\partial F(x^k))>\epsilon$ \textbf{do}
  \begin{enumerate}
  		\item  Choose $L_{k,0}\in[L_{\min},L_{\max}]$.
		
		\item  {\bf For} $j=0,1,2,\ldots$ {\bf do} 
		\begin{enumerate}
			\item Compute a stationary point $x^{k,j}$ with $\widetilde{F}_{k,j}(x^{k,j})\le \widetilde{F}_{k,j}(x^{k})$ for the subproblem 
			\begin{equation}\label{subprob1}
				\min_{x\in\mathbb{X}}\widetilde{F}_{k,j}(x):=f_k(x)+(L_{k,j}/q)\|x-x^{k}\|^q+g(x).
			\end{equation}
			
			\item If $F(x^{k,j})\le F(x^k)-(\delta/q)L_{k,j}\|x^{k,j}-x^k\|^q$, set $j_k=j$ and go to step 4. Otherwise, let $L_{k,j+1}=\tau^{j+1} L_{k,0}$. 
		\end{enumerate}
		\item  {\bf End (for)}    
		
		\item  Set $L_{k}=L_{k,j_k},x^{k+1}=x^{k,j_k}$ and $\widetilde{F}_k=\widetilde{F}_{k,j_k}$, and let $k\leftarrow k+1$.
  \end{enumerate}
 \textbf{End (while)}
 \end{algorithm}
\begin{remark}\label{Alg-remark}
 {\bf(a)} Note that $x$ is a stationary point of problem \eqref{composite} whenever $0\in\partial F(x)=\!\nabla\!f(x)+\partial g(x)$. This inspired us to adopt ${\rm dist}(0,\partial F(x^k))\le\epsilon$ as the stopping condition of Algorithm \ref{cubic-reg}, and the final output provides an approximate stationary point of this sense. Such a stopping condition can be checked without requiring the characterization of $\partial F(x^{k+1})$. Indeed, from the stationary condition of \eqref{subprob1}, the inclusion \eqref{aim-ineq2} in the proof of Lemma \ref{pro-alg1} holds, which means that such a stopping condition is satisfied at $x^{k+1}$ if $\|\nabla\!f(x^{k+1})\!-\!\nabla\!f(x^{k})\!-\!\nabla^2\!f(x^{k})(x^{k+1}-\!x^k)
 	-L_k\|x^{k+1}-x^k\|^{q-2}(x^{k+1}\!-\!x^k)\|\le\epsilon$.
 
 \noindent
 {\bf(b)} When $\nabla^2f$ is strictly continuous at $x^k$, the constant $L_{k,0}$ in step 1 is an initial estimation for the Lipschitz modulus of the Hessian of $f$ around $x^{k}$. A good initialization for it will reduce the cost of the inner for-end loop. Inspired by the technique to estimate the Lipschitz constant of $\nabla\!f$ in \cite{Wright09}, we suggest the Barzilai-Borwein (BB) rule \cite{Barzilai88} to capture a desirable $L_{k,0}$. 
 
 \noindent
 {\bf(c)} By the expression of $f_k$, when $L_{k,j}$ is sufficiently large, the function $\widetilde{F}_{k,j}$ is coercive by Assumption \ref{ass-F} (ii) and its global minimizer exists. Let $\overline{x}^{k,j}$ be a global minimizer of $\widetilde{F}_{k,j}$. Clearly, $\widetilde{F}_{k,j}(\overline{x}^{k,j})<\widetilde{F}_{k,j}(x^k)$ (if not, $x^k$ is a global minimizer of $\widetilde{F}_{k,j}$ with $0\in\partial\widetilde{F}_{k,j}(x^k)=\partial F(x^k)$, and Algorithm \ref{cubic-reg} stops at $x^k$). From the continuity of $F$ relative to ${\rm dom}\,g$ by Assumption \ref{ass-F}, $x^{k,j}$ in step (2a) can take any stationary point $y$ sufficiently close to $\overline{x}^{k,j}$. Thus, we conclude that Algorithm \ref{cubic-reg} is well defined. 
  	
 \noindent
 {\bf(d)} Since the Hessian $\nabla^2\!f$ is not assumed to be globally Lipschitz on a closed convex set containing all iterates as in the reference (see, e.g., \cite{Yue19,Song19,Zhou18}), and even when it is strictly continuous at $x^k$, its Lipschitz modulus at $x^k$ is usually unknown, now step (2a) aims to search a desirable estimation $L_{k}$ for it. 
\end{remark}
\begin{lemma}\label{ls-welldef}
 Under Assumption \ref{ass-F}, if Algorithm \ref{cubic-reg} does not terminate at the $k$th iterate, then its inner loop must stop within a finite number of steps. 
\end{lemma}
 \begin{proof} Assume that ${\rm dist}(0,\partial F(x^k))>\epsilon$. Suppose on the contrary that its inner loop does not stop within a finite number of steps, i.e., for each $j\in\mathbb{N}$, 
 \begin{equation}\label{aim-ineq1}
  F(x^{k,j})-F(x^k)>-(\delta/q)L_{k,j}\|x^{k,j}-x^k\|^q,
 \end{equation} 
 From step (2a) and the expression of $\widetilde{F}_{k,j}$, it follows that for each $j\in\mathbb{N}$, 
 \begin{equation}\label{wd-ineq1}
 f_k(x^{k,j})+g(x^{k,j})+({L_{k,j}}/{q})\|x^{k,j}-x^{k}\|^q\le F(x^{k}).
 \end{equation}
 Since $L_{k,j}\!\to\infty$ as $j\!\to\infty$, $\langle x^{k,j}\!-\!x^k,\nabla^2\!f(x^k)(x^{k,j}\!-\!x^k)\rangle+\frac{L_{k,j}}{q}\|x^{k,j}\!-\!x^{k}\|^q\to \infty$ as $j\to\infty$. Along with the expression of $f_k$ and the lower boundedness of $g$ by
 Assumption \ref{ass-F} (ii), the last inequality implies that $x^{k,j}\to x^k$ as $j\to\infty$.
 Note that $\{x^{k,j}\}_{j\in\mathbb{N}}\subset{\rm dom}\,g$. By Assumption \ref{ass-F} (ii),  $\lim_{j\to\infty}g(x^{k,j})= g(x^k)$. Next we claim that $\liminf_{j\to\infty}L_{k,j}\|x^{k,j}-x^k\|^{q-1}>0$. If not, there is an index set $J\subset\mathbb{N}$ such that $\lim_{J\ni j\to\infty}L_{k,j}\|x^{k,j}-x^k\|^{q-1}=0$. By the definition of $x^{k,j}$, $ 0\in\nabla\!f(x^{k})+\nabla^2\!f(x^{k})(x^{k,j}\!-\!x^k)+L_{k,j}\|x^{k,j}\!-\!x^k\|^{q-2}(x^{k,j}\!-\!x^k)+\partial g(x^{k,j})$, which along with $\partial F(x)=\nabla f(x)+\partial g(x)$ for $x\in{\rm dom}\,g$ is equivalent to saying that 
 \[
  \nabla\!f(x^{k,j})-\nabla\!f(x^{k})-\nabla^2\!f(x^{k})(x^{k,j}-x^k)
  -L_{k,j}\|x^{k,j}-x^k\|^{q-2}(x^{k,j}-x^k)\!\in\partial F(x^{k,j}).
 \]
 Because $\nabla\!f$ is continuous on ${\rm dom}\,g$ by Assumption \ref{ass-F} (i),  passing the limit $J\ni j\to\infty$ and using $x^{k,j}\!\to x^k$ as $j\to\!\infty$ and the outer semicontinuity of $\partial F$ yields that $0\in\partial F(x^k)$, which is impossible because ${\rm dist}(0,\partial F(x^k))>\epsilon$. Thus,  $\liminf_{j\to\infty}L_{k,j}\|x^{k,j}\!-x^k\|^{q-1}>0$, so there exists $\gamma>0$ such that  $L_{k,j}\|x^{k,j}\!-\!x^k\|^{q-1}\ge \gamma$ for all sufficiently large $j$, and then
 \begin{equation}\label{temp-ineq41}
  (1-\delta)\frac{L_{k,j}}{q}\|x^{k,j}-x^k\|^q
  \ge\frac{1-\delta}{q}\gamma\|x^{k,j}-x^k\|\ge o(\|x^{k,j}-x^k\|^2).
 \end{equation}
 From the twice differentiability of $f$ on ${\rm dom}\,g$, for all sufficiently large $j\in\mathbb{N}$, 
 \begin{align*}
  F(x^{k,j})-F(x^k)&=f_k(x^{k,j})+g(x^{k,j})-g(x^k)+o(\|x^{k,i}-x^k\|^2)\\
   &\le -({L_{k,j}}/{q})\|x^{k,j}-x^k\|^q+o(\|x^{k,j}-x^k\|^2)\\
   &\le -\delta({L_{k,j}}/{q})\|x^{k,j}-x^k\|^q,
 \end{align*}
 where the first inequality is using \eqref{wd-ineq1}, and the last one is due to \eqref{temp-ineq41}. The last inequality is a contradiction to \eqref{aim-ineq1}. The conclusion then follows. 
 \end{proof}
 
 For the sequence generated by Algorithm \ref{cubic-reg}, the following conclusion holds. 
 \begin{lemma}\label{pro-alg1}
  Let $\{x^k\}_{k\in\mathbb{N}}$ be the sequence generated by Algorithm \ref{cubic-reg} with $\epsilon=0$. Suppose that Assumption \ref{ass-F} holds and that $\mathcal{L}_{F(x^0)}:=\{x\in\mathbb{X}\,|\,F(x)\le F(x^0)\}$ is bounded. Then the following statements hold. 
 \begin{description}
  \item [(i)] For each $k\in\mathbb{N}$, $F(x^{k+1})\le F(x^k)-\delta(L_{\rm min}/q)\|x^{k+1}-x^k\|^q$. 
		
  \item[(ii)] $\{F(x^{k})\}_{k\in\mathbb{N}}$ is nonincreasing and convergent, and ${\displaystyle\lim_{k\to\infty}}\|x^{k}-x^{k-1}\|=0$.
		
  \item [(iii)] The sequence $\{x^k\}_{k\in\mathbb{N}}$ is bounded and its accumulation point set, denoted by $\omega(x^0)$, is nonempty and compact. 
  
  \item[(iv)] If in addition $\nabla^2\!f$ is strictly continuous on an open neighborhood $\mathcal{N}$ of $\omega(x^0)$, the sequence $\{L_k\}_{k\in\mathbb{N}}$ is bounded and for each $k\in\mathbb{N}$ there exists $w^{k+1}\in\partial F(x^{k+1})$ such that $\|w^{k+1}\|\le\alpha\|x^{k+1}\!-x^{k}\|^{q-1}$ for some $\alpha>0$. 
  
  \item[(v)] There exists a subsequence $\{x^{k_j}\}_{j\in\mathbb{N}}$ with 
  $x^{k_j}\!\to\widetilde{x}\in\omega(x^0)$ and $F(x^{k_j})\to F(\widetilde{x}) $ as $j\to\infty$. 
 \end{description}
 \end{lemma}
 \begin{proof}
 Part (i) is trivial by the inner loop of Algorithm \ref{cubic-reg} and $L_{k,j}\ge L_{\rm min}$ for each $k\in\mathbb{N}$ and $j\in\mathbb{N}$. From part (i), the sequence $\{F(x^k)\}_{k\in\mathbb{N}}$ is nonincreasing, which implies that $\{x^k\}_{k\in\mathbb{N}}\subset \mathcal{L}_{F(x^0)}$. Together with the boundedness of $\mathcal{L}_{F(x^0)}$, part (iii) then follows. Since $\{x^k\}_{k\in\mathbb{N}}$ is bounded, the sequence $\{f(x^k)\}_{k\in\mathbb{N}}$ has a lower bound by Assumption \ref{ass-F} (i), which by the lower boundedness of $g$ in Assumption \ref{ass-F} (ii) means that $\{F(x^k)\}_{k\in\mathbb{N}}$ has a lower bound. Thus, $\{F(x^k)\}_{k\in\mathbb{N}}$ is convergent and $\lim_{k\to\infty}\|x^{k}-x^{k-1}\|=0$. 
 
 \noindent
 {\bf(iv)} Suppose on the contradiction that the sequence $\{L_k\}_{k\in\mathbb{N}}$ is unbounded. Then, there necessarily exists an index set $K\!:=\!\{k\in\mathbb{N}\ |\ j_k\ge 1\}$ such that $\lim_{K\ni k\to\infty}L_{k}=\infty$.  For each $k\in K$, write $\widetilde{L}_{k}:=L_k/\tau$. From Lemma \ref{ls-welldef} and the inner loop of Algorithm \ref{cubic-reg}, for each $k\in K$, 
 \begin{equation}\label{Lbounded-ineq1}
  F(x^{k,j_k-1})>F(x^{k})-(\delta/q)\widetilde{L}_{k}\|x^{k,j_k-1}-x^{k}\|^q.
 \end{equation}
 On the other hand, from the definition of $x^{k,j_k-1}$, for each $k\in K$ it holds that 
  \begin{equation}\label{Lbounded-ineq2}
  f_k(x^{k,j_k-1})+g(x^{k,j_k-1})+(\widetilde{L}_{k}/q)\|x^{k,j_k-1}\!-\!x^{k}\|^q-F(x^{k})\le 0
  \end{equation}
  which, along with the nonincreasing of $\{F(x^k)\}_{k\in\mathbb{N}}$ by part (i), implies that 
  \begin{align*}
  (\widetilde{L}_{k}/q)\|x^{k,j_k-1}\!-\!x^{k}\|^q
  &\le F(x^0)-F(x^{k,j_k-1})+\|\nabla\!f(x^k)\|\|x^{k,j_k-1}\!-x^{k}\|\nonumber\\
  &\quad +\|\nabla^2\!f(x^k)\|\|x^{k,j_k-1}\!-\!x^{k}\|^2.
  \end{align*}
  This, by the lower boundedness of $\{F(x^{k,j_k-1})\}$ due to Assumption \ref{ass-F} and $\lim_{K\ni k\to\infty}\widetilde{L}_{k}=\infty$, implies that $\lim_{K\ni k\to\infty}\|x^{k,j_k-1}\!-\!x^{k}\|=0$. If necessary by taking a subsequence, we assume that $\lim_{K\ni k\to\infty}x^{k}=x^*\in\omega(x^0)$. Along with $\lim_{K\ni k\to\infty}\|x^{k,j_k-1}\!-\!x^{k}\|=0$, $\lim_{K\ni k\to\infty}x^{k,j_k-1}=x^*$. Now from Assumption \ref{ass-F} (i) and the mean-valued theorem, for each $k\in K$, there exists $t_{k}\in(0,1)$ such that for $y^{k}:=x^{k}+t_{k}(x^{k,j_k-1}-x^{k})$,
  \begin{align*}
   f(x^{k,j_k-1})-f(x^{k})&=\langle\nabla\!f(x^{k}),x^{k,j_k-1}-x^{k} \rangle \\
	&\quad +\langle\nabla^2\!f(y^{k})(x^{k,j_k-1}-x^{k}),x^{k,j_k-1}-x^{k} \rangle. 
  \end{align*}
  This, along with inequality \eqref{Lbounded-ineq2}, implies that for each $k\in K$, 
  \begin{align}\label{Lbounded-ineq3}
  &({\widetilde{L}_{k}}/{q})\|x^{k,j_k-1}\!-\!x^{k}\|^q
		\le -f_k(x^{k,j_k-1})-g(x^{k,j_k-1})+F(x^{k})\nonumber\\
  &=f(x^{k,j_k-1})-f(x^{k})-f_k(x^{k,j_k-1})+f(x^k)+F(x^k)-F(x^{k,j_k-1})\nonumber\\	
  &=\langle(\nabla^2\!f(y^{k})\!-\!\nabla^2\!f(x^{k}))(x^{k,j_k-1}\!-x^{k}),
		x^{k,j_k-1}\!-x^{k}\rangle+F(x^k)\!-\!F(x^{k,j_k-1})\nonumber\\
  &<\|\nabla^2\!f(y^{k})-\nabla^2\! f(x^{k})\|
		\|x^{k,j_k-1}-x^{k}\|^2+\delta(\widetilde{L}_{k}/q)\|x^{k,j_k-1}\!-\!x^{k}\|^q
  \end{align}
  where the last inequality is due to \eqref{Lbounded-ineq1}. Note that $\lim_{K\ni k\to\infty}y^{k}=x^*$. From the strict continuity of $\nabla^2\!f$ at $x^*$, there exists $L>0$ such that for all sufficiently large $k\in K$, 
  \[
	\|\nabla^2\!f(y^{k})-\nabla^2\!f(x^{k})\|\le L\|x^{k,j_k-1}-x^{k}\|,
  \] 
  which together with \eqref{Lbounded-ineq3} implies that for all sufficiently large $k\in K$, 
  \[
	(1-\delta)(\widetilde{L}_{k}/q)\|x^{k,j_k-1}\!-\!x^{k}\|^q
	\le L\|x^{k,j_k-1}\!-\!x^{k}\|^3.
  \]
  This is impossible by recalling that $\lim_{K\ni k\to\infty}\widetilde{L}_{k}=\infty, x^{k,j_k-1}\neq x^{k}$ for each $k\in K$ by \eqref{Lbounded-ineq1}, and $q\in[2,3]$. Consequently, the sequence $\{L_{k}\}_{k\in\mathbb{N}}$ is bounded. 
	
  We next prove the rest conclusions. For each $k$, by the definition of $x^{k+1}$, 
  \[
	0\in\nabla\!f(x^{k})+\nabla^2\!f(x^{k})(x^{k+1}\!-\!x^k)
	+L_{k}\|x^{k+1}\!-\!x^k\|^{q-2}(x^{k+1}\!-\!x^k)+\partial g(x^{k+1}), 
  \]
  which along with the expression of $F$ implies that for each $k\in\mathbb{N}$,
  \begin{align}\label{aim-ineq2}
   \partial F(x^{k+1})\ni w^{k+1}&\!:=\nabla\!f(x^{k+1})\!-\!\nabla\!f(x^{k})\!-\!\nabla^2\!f(x^{k})(x^{k+1}-\!x^k)\nonumber\\
   &\qquad-\!L_k\|x^{k+1}-x^k\|^{q-2}(x^{k+1}\!-\!x^k). 
	\end{align}
   We claim that there is $\overline{L}>0$ such that for all $k\in K_1:=\{k\in\mathbb{N}\ |\ x^{k+1}\ne x^k\}$,
  \begin{equation}\label{aim-ineq3}
   \limsup_{K_1\in k\to\infty}\frac{\|\nabla\!f(x^{k+1})\!-\!\nabla\!f(x^{k})\!-\!\nabla^2\!f(x^{k})(x^{k+1}-\!x^k)\|}{\|x^{k+1}\!-\!x^k\|^{q-1}}\le \overline{L}.
  \end{equation}
  If not, there will exist an index set $K_2\subseteq K_1$ such that 
  \begin{equation}\label{temp-ineq}
  \lim_{K_2\ni k\to\infty}\frac{\|\nabla\!f(x^{k+1})\!-\!\nabla\!f(x^{k})\!-\!\nabla^2\!f(x^{k})(x^{k+1}-\!x^k)\|}{\|x^{k+1}\!-\!x^k\|^{q-1}}=\infty.
  \end{equation}
  On the other hand, by Assumption \ref{ass-F} (i) and the mean-valued theorem, for each $k\in\mathbb{N}$ there exists $t_k\in(0,1)$ such that
  \begin{equation*}
   \nabla\!f(x^{k+1})\!-\!\nabla\!f(x^{k})
	=\nabla^2\!f(x^{k}\!+\!t_k(x^{k+1}\!-\!x^k))(x^{k+1}\!-\!x^k).
   \end{equation*} 
  We assume that $\lim_{K_2\ni k\to\infty}x^k=\widehat{x}^*$ (if necessary by  taking a subsequence of $\{x^k\}_{k\in K_2}$). Clearly, $\widehat{x}^*\in\omega(x^0)$. Moreover, by part (ii), $\lim_{K_2\ni k\to\infty}x^{k}+t_k(x^{k+1}\!-\!x^k)=\widehat{x}^*$. Since $\nabla^2\!f$ is strictly continuous at $\widehat{x}^*$, there exists $\widehat{L}>0$ such that for all $k\in K_2$ large enough, 
  \[
	\|\nabla^2\!f(x^{k}+t_k(x^{k+1}\!-\!x^k))-\nabla^2\!f(x^{k})\|
	\le\widehat{L}\|x^{k+1}-x^k\|.
   \]
  Combining the last two equations with $\lim_{k\to\infty}\|x^{k+1}-x^k\|=0$ and using $q\in[2,3]$ yields that
   \[
	\lim_{K_2\ni k\to\infty}\frac{\|\nabla\!f(x^{k+1})\!-\!\nabla\!f(x^{k})\!-\!\nabla^2\!f(x^{k})(x^{k+1}-\!x^k)\|}{\|x^{k+1}\!-\!x^k\|^{q-1}}
	\le\widehat{L},
  \]
  which is a contradiction to \eqref{temp-ineq}. Consequently, the claimed inequality \eqref{aim-ineq3} holds, and there exists a constant $\widetilde{L}>0$ such that for all $k\in K_1$
  \[
	|\nabla\!f(x^{k+1})\!-\!\nabla\!f(x^{k})\!-\!\nabla^2\!f(x^{k})(x^{k+1}-\!x^k)\|
	\le \widetilde{L}\|x^{k+1}\!-\!x^k\|^{q-1}.
  \]
  Along with \eqref{aim-ineq2} and the definition of $K_1$, $w^{k+1}\!\in\partial F(x^{k+1})$ for each $k\in\mathbb{N}$ with $\|w^{k+1}\|\le(L_k+\widetilde{L})\|x^{k+1}-x^{k}\|^{q-1}$. 
  Recall that the sequence $\{L_k\}_{k\in\mathbb{N}}$ is bounded, so $\|w^{k+1}\|\le\alpha\|x^{k+1}-x^{k}\|^{q-1}$ with some $\alpha>0$ for all $k\in\mathbb{N}$. 
  
  \noindent
  {\bf(v)} The result follows by the continuity of $g$ relative to its domain. 
 \end{proof}
 
 Lemma \ref{pro-alg1} (ii) and (iv)-(v) show that the sequence generated by Algorithm \ref{cubic-reg} satisfies conditions H1-H3 with $\Phi=F$ and $p=q-1$. Thus, by invoking Theorems \ref{KL-converge}-\ref{KL-rate}, we immediately have the following conclusion. 
\begin{theorem}\label{theorem-Alg1}
  Let $\{x^k\}$ be the sequence generated by Algorithm \ref{cubic-reg} with $\epsilon=0$. Suppose that Assumption \ref{ass-F} holds and that the set $\mathcal{L}_{F(x^0)}$ is bounded. Then,
 \begin{description}
 \item [(i)] when $F$ is a KL function, $\sum_{k=0}^{\infty}\|x^k\!-\!x^{k-1}\|<\infty$; 
		
 \item [(ii)] when $F$ is a KL function of exponent $\theta\in\!(0,\frac{q-1}{q})$,    $\{x^k\}_{k\in\mathbb{N}}$ converges Q-superlinearly with order $\frac{q-1}{\theta q}$ to a  stationary point of \eqref{composite};
		
 \item [(iii)] when $F$ is a KL function of exponent $\theta\in[\frac{q-1}{q},1)$, there exist $\gamma>0$ and $\varrho\in[0,1)$ such that 
  \begin{equation}
  \|x^k-\widetilde{x}\|\le\left\{\begin{array}{cl}
   \gamma\varrho^{k} &{\rm if}\ \theta=\frac{q-1}{q},\\
   \gamma{k}^{\frac{1-\theta}{1-(q\theta)/(q-1)}}&{\rm if}\ \theta\in(\frac{q-1}{q},1).
			\end{array}\right.
 \end{equation}
 \end{description}
\end{theorem}
\begin{remark}
 Theorem \ref{theorem-Alg1} (ii) implies that the sequence generated by the CR method converges Q-superlinearly with order $4/3$ to a stationary point of \eqref{composite} if $F$ is a KL function of exponent $1/2$ satisfying Assumption \ref{ass-F} and the set $\mathcal{L}_{F(x^0)}$ is bounded. When $g\equiv 0$, although this result is weaker than the Q-quadratic superlinear convergence rate obtained in \cite{Yue19}, the required KL property of exponent $1/2$ is also weaker than the local error bound required by the latter on the second-order stationary point set. 
\end{remark} 

\section{Conclusions}

 We conducted a systematic analysis on the convergence of sequences complying with conditions H1-H3, a more general iterative framework than the one studied in \cite{Attouch13} for nonconvex and nonsmooth KL optimization, and derived its Q-superlinear convergence rate of order $\frac{p}{\theta(1+p)}$ for the KL function $\Phi$ with exponent $\theta\in(0,\frac{p}{p+1})$. When $p=1$, this sharpens the R-linear convergence result in \cite[Theorem 1]{Attouch09} for the KL function $\Phi$ of exponent $\theta\in(0,{1}/{2})$. We provided a $q\in[2,3]$-order regularized proximal Newton method for composite optimization problems with a nonconvex and nonsmooth term, whose iterate sequence was shown to fall into this framework, and then first achieved a Q-superlinear convergence rate of order $4/3$ for a cubic regularization method to this class of composite KL optimiztion of exponent $1/2$. 

\medskip
\noindent
{\large\bf Acknowledgements.}
This work is funded by the National Natural Science Foundation of China under project No. 12371299. 



\begin{thebibliography}{}
 	
 \bibitem{Attouch09} 
 Attouch, H., Bolte, J.: 
 	On the convergence of the proximal algorithm for nonsmooth functions involving analytic features. 
 	Mathematical Programming, 116, 5-16 (2009)
	
	 \bibitem{Attouch10}
	 Attouch, H., Bolte, J., Redont, P., Soubeyran, A.:
	 Proximal alternating minimization and projection methods for nonconvex problems: an approach based on the Kurdyka-{\L}ojasiewicz inequality.
	 Mathematics of Operations Research, 35, 438-457 (2010)

	\bibitem{Attouch13}
	Attouch, H., Bolte, J., Svaiter, B.F.:
	Convergence of descent methods for semi-algebraic and tame problems:
	proximal algorithms, forward-backward splitting, and regularized Gauss-Seidel methods.
	Mathematical Programming, 137, 91-129 (2013)
	
	\bibitem{Barzilai88}
	Barzilai, J., Borwein, J.M.:
	Two-point step size gradient methods.
	IMA Journal of Numerical Analysis, 8, 141-148 (1988)
	
	\bibitem{Bolte14}
	Bolte, J., Sabach, S., Teboulle, M.:
	Proximal alternating linearized minimization for nonconvex and nonsmooth problems.
	Mathematical Programming, 146, 459-494 (2014)
	
	
	\bibitem{1Cartis11}
	Cartis, C., Gould, N.I.M., Toint, Ph.L.:
	Adaptive cubic regularisation methods for unconstrained optimization. 
	Part I: Motivation, convergence and numerical results. 
	Mathematical Programming, 127, 245-295 (2011)
	
	\bibitem{2Cartis11}
	Cartis, C., Gould, N.I.M., Toint, Ph.L.:
	Adaptive cubic regularisation methods for unconstrained optimization. 
	Part II: Worst-case function- and derivative-evaluation complexity.
	Mathematical Programming, 130, 295-319 (2011)
	
	\bibitem{Doikov22}
	Doikov, N., Nesterov, Y.:
	Local convergence of tensor methods.
	Mathematical Programming, 193, 315-336 (2022)
	
	\bibitem{Grapiglia19}
	Grapiglia, G.N., Nesterov, Y.:
	Accelerated regularized Newton methods for minimizing composite convex functions.
	SIAM Journal on Optimization, 29, 77-99 (2019)
	
	\bibitem{Griewank81}
	Griewank, A.:
	The modification of Newton's method for unconstrained optimization by
	bounding cubic terms.
	Technical report, Department of Applied Mathematics and Theoretical
	Physics, University of Cambridge, Cambridge, UK (1981)
	
	\bibitem{Jiang20}
	Jiang, B., Lin, T.Y., Zhang, S.Z.:
	A unified adaptive tensor approximation scheme to accelerate composite convex optimization. 
	SIAM Journal on Optimization, 30, 2897-2926 (2020) 
	
	\bibitem{Kanzow21}
	Kanzow, C., Lechner, T.:
	Globalized inexact proximal Newton-type methods for nonconvex composite functions. 
	Computational Optimization and Applications, 78, 377-410 (2021)
	
	\bibitem{Lee14}
	Lee, J.D., Sun, Y.K., Saunders, M.A.: 
	Proximal Newton-type methods for minimizing composite functions.
	SIAM Journal on Optimization, 24, 1420-1443 (2014)
	
	\bibitem{LiPong18}
	Li, G.Y., Pong, T.K.:
	Calculus of the exponent of Kurdyka-{\L}ojasiewicz inequality and its applications
	to linear convergence of first-order methods.
	Foundations of Computational Mathematics, 18, 1199-1232 (2018)
	
	\bibitem{Mordu06}
	Mordukhovich, B.S.:
	Variational Analysis and Generalized Differentiation, 
	I: Basic Theory; II: Applications.
	Springer, New York (2006) 
		
	\bibitem{Mordu22}
	Mordukhovich, B.S., Yuan, X.M., Zeng, S.Z., Zhang, J.:
	A globally convergent proximal {N}ewton-type method in nonsmooth convex optimization.
	Mathematical Programming, https://doi.org/10.1007/s10107-022-01797-5 (2022)
	
	\bibitem{Nabou23}
	Nabou, Y., Necoara, I.:
	Efficiency of higher-order algorithms for minimizing general composite optimization. Computational Optimization and Applications, \url{https://doi.org/10.1007/s10589-023-00533-9}, 2023. 
	
	\bibitem{Necoara21}
	Necoara, I., Lupu, D.:
	General higher-order majorization-minimization algorithms for (non)convex optimization.
	arXiv:2010.13893v3 (2021) 
	
	\bibitem{Nesterov06}
	Nesterov, Y., Polyak, B.T.:
	Cubic regularization of Newton method and its global performance.
	Mathematical Programming, 108, 177-205 (2006) 
	
	\bibitem{Nesterov08}
	Nesterov, Y.:
	Accelerating the cubic regularization of Newton's method on convex problems.
	Mathematical Programming, 112, 159-181 (2008)
	
	
	\bibitem{Nesterov22}
	Nesterov, Y.:
	Inexact basic tensor methods for some classes of convex optimization problems.
	Optimization Methods and Software, 37, 878-906 (2022)
	
	\bibitem{Ochs18}
	Ochs, P.:
	Local convergence of the heavy-ball method and iPiano for non-convex optimization.
	Journal of Optimization Theory and Methods, 177, 153-180 (2018)
	
	\bibitem{RW98}
	Rockafellar, R.T., Wets, R.J-B.:
	Variational Analysis. Springer (1998) 
	
	\bibitem{Song19}
	Song, C.B., Liu, J., Yong, J.:
	Inexact proximal cubic regularized Newton methods for convex optimization. 
	10.48550/arXiv.1902.02388 (2019)
	
	\bibitem{Wright09}
	Wright, S.J., Nowak, R., Figueiredo, M.:
	Sparse reconstruction by separable approximation.
	IEEE Transactions on Signal Processing, 57, 2479-2493 (2009)
	
	\bibitem{WuPanBi21}
	Wu, Y.Q., Pan, S.H., Bi, S.J.:
	Kurdyka-{\L}ojasiewicz property of zero-norm composite functions.
	Journal of Optimization Theory and Applications, 188, 94-112 (2021)
	
	\bibitem{YuLiPong21}
	Yu, P.R., Li, G.Y., Pong, T.K.:
	Kurdyka-{\L}ojasiewicz exponent via inf-projection.
	Foundations of Computational Mathematics, 22, 1171-1217 (2022)
	
	\bibitem{Yue191}
	Yue, M.C., Zhou, Z.R., So, A.M.-C.:
	A family of inexact SQA methods for non-smooth convex minimization with provable convergence guarantees based on the Luo-Tseng error bound property. 
	Mathematical Programming, 174, 327-358 (2019)
	
	\bibitem{Yue19}
	Yue, M.C., Zhou, Z.R., So, A.M.-C.:
	On the quadratic convergence of the cubic regularization method under a local error bound condition. 
	SIAM Journal on Optimization, 29, 904-932 (2019)
	
	\bibitem{Zhou18}
	Zhou, Y., Wang, Z., Liang, Y.B.:
	Convergence of cubic regularization for nonconvex optimization under KL property.
	Neural Information Processing Systems Conference (2018)
	
	\end{thebibliography}
\end{document}